\documentclass[12pt]{article}
\usepackage{amsmath}
\usepackage{amsfonts}
\usepackage{amsthm}
\usepackage{amssymb}
\usepackage{color}
\usepackage{enumerate}
\usepackage[mathscr]{euscript}

\newtheorem{theorem}{Theorem}[section]
\newtheorem{lemma}[theorem]{Lemma}
\newtheorem{proposition}[theorem]{Proposition}
\newtheorem{observation}[theorem]{Observation}
\newtheorem{problem}[theorem]{Problem}
\newtheorem{dir}[theorem]{Direction}
\newtheorem{example}[theorem]{Example}
\newtheorem{corollary}[theorem]{Corollary}

\theoremstyle{definition}
\newtheorem{definition}[theorem]{Definition}
\theoremstyle{remark}
\newtheorem{remark}[theorem]{Remark}

\newcommand{\kr}{{\rm KR}\hskip0.02cm}
\newcommand{\tp}{{\rm TP}\hskip0.02cm}
\newcommand{\tc}{{\rm TC}\hskip0.02cm}

\def\mf{\mathcal{F}}

\def\mk{\mathcal{K}}
\def\ml{\mathcal{L}}
\def\mn{\mathcal{N}}
\def\mr{\mathcal{R}}
\def\ms{\mathcal{S}}
\def\mdp{\mathcal{DP}}

\def\f2{\mathbb{F}_2}

\def\lip{\hskip0.02cm{\rm Lip}\hskip0.01cm}
\def\supp{\hskip0.02cm{\rm supp}\hskip0.01cm}

\newcommand{\Bs}{\mathscr{B}}

\newcommand{\ep}{\varepsilon}

\newcommand{\sign}{{\rm sign}\hskip0.02cm}

\newcommand{\1}{\mathbf{1}}

\newcommand\remove[1]{}

\begin{document}

\title{\LARGE Generalized transportation cost spaces}

\author{Sofiya Ostrovska and Mikhail~I.~Ostrovskii}

\date{\today}
\maketitle



\begin{abstract} The paper is devoted to the geometry of transportation cost
spaces  and their generalizations introduced by Melleray, Petrov,
and Vershik (2008). Transportation cost spaces are also known as
Arens-Eells, Lipschitz-free, or Wasserstein $1$ spaces. In this
work, the existence of metric spaces with the following properties
is proved: (1) uniformly discrete infinite metric spaces
transportation cost spaces on which do not contain isometric
copies of $\ell_1$, this result answers a question raised by
C\'uth and Johanis (2017); (2) locally finite metric spaces which
admit isometric embeddings only into Banach spaces containing
isometric copies of $\ell_1$; (3) metric spaces for which the
double-point norm is not a norm. In addition, it is proved that
the double-point norm spaces corresponding to trees are close to
$\ell_\infty^d$ of the corresponding dimension, and that for all
finite metric spaces $M$, except a very special class, the infimum
of all seminorms for which the embedding of $M$ into the
corresponding seminormed space is isometric, is not a seminorm.
\end{abstract}



\section{Introduction}

\subsection{Definitions}\label{S:Def}

 Let $(M,d)$ be a metric space. Consider a real-valued finitely
supported function $f$ on $M$ with a zero sum, that is,

\begin{equation}\label{E:Sum0}\sum_{v\in M}f(v)=0.\end{equation}

A natural and important interpretation of such a function is the
following: $f(v)>0$ means that $f(v)$ units of a certain product
are produced or stored at point $v$; $f(v)<0$ means that $(-f(v))$
units of the same product are needed at $v$.  The number of units
can be any real number.  With this in mind, $f$ may be regarded as
a {\it transportation problem}. For this reason, we denote the
vector space of all real-valued functions finitely supported on
$M$ with a zero sum by $\tp(M)$, where $\tp$ stands for {\it
transportation problems}.
\medskip

For a metric space $M$ with the {\it base point}, which is a
distinguished point usually denoted by $O$, there is a {\it
canonical embedding}\,  of $M$ into $\tp(M)$ given by the formula:
\begin{equation}\label{E:EmbBasePt} v\mapsto \1_v-\1_O,\end{equation}
where $\1_u(x)$ for $u\in M$ is the {\it indicator function}
defined as:
\[\1_u(x)=\begin{cases} 1 &\hbox{ if }x=u,\\ 0 &\hbox{ if }x\ne u.
\end{cases} \]

The goal of this work  is to study different norms on the vector
space $\tp(M)$ for which this embedding is an isometric embedding.
\medskip

One of the most commonly used norms on $\tp(M)$ satisfying this
condition is that related to the {\it transportation cost} and
defined in the following way.\medskip

A {\it transportation plan} is a plan of the following type: we
intend to deliver

\begin{itemize}

\item $a_1$ units of the product from $x_1$ to $y_1$,

\item $a_2$ units of the product from $x_2$ to $y_2$,

\item \dots

\item $a_n$ units of the product from $x_n$ to $y_n$,

\end{itemize}
where $a_1,\dots,a_n$ are nonnegative real numbers, and
$x_1,\dots,x_n,y_1,\dots,y_n$ are elements of $M$, which do not
have to be distinct.\medskip

This transportation  plan is said to  {\it solve the
transportation problem} $f$ if
\begin{equation}\label{E:TranspPlan}f=a_1(\1_{x_1}-\1_{y_1})+a_2(\1_{x_2}-\1_{y_2})+\dots+
a_n(\1_{x_n}-\1_{y_n}).\end{equation}

 The {\it cost} of  transportation plan
\eqref{E:TranspPlan} is
 defined as $\sum_{i=1}^n a_id(x_i,y_i)$. We introduce the
{\it transportation cost norm} (or just {\it transportation cost})
$\|f\|_{\tc}$ of a transportation problem $f$ as the minimal cost
of transportation plans solving $f$. It is easy to see that the
minimum is attained - we consider finitely supported functions -
and that $\|\cdot\|_\tc$ is a norm (see \cite[Proposition
3.16]{Wea18}). We introduce the {\it transportation cost space}
$\tc(M)$ on $M$ as the completion of $\tp(M)$ with respect to the
norm $\|\cdot\|_\tc$. The introduced above notions are very
natural and were introduced independently and not-so-independently
by many different people, whence a variety of names. There are
also very important and actively studied notions, which are
somewhat different from the introduced above, but are closely
related to them. A survey  of these definitions, relations between
them, and some historical notes are provided in Section
\ref{S:HistTerm}.

Our main reasons for choosing the term {\it transportation cost
space} are: (1) This term will make it immediately clear to as
many people as possible what is the topic of this paper; (2) This
terminology helps us to develop a suitable language and to build
the right intuition for working with the norm $\|\cdot\|_\tc$; (3)
It reflects the history of the subject and the initial motivation
for introducing these notions.

\subsection{Preliminaries}\label{S:MoreNorms}

\begin{theorem}[\cite{KG49}]\label{T:WitnLip} A plan
\begin{equation}f=a_1(\1_{x_1}-\1_{y_1})+a_2(\1_{x_2}-\1_{y_2})+\dots+
a_n(\1_{x_n}-\1_{y_n})\end{equation} with $a_i>0$, $i=1,\dots,n$,
is optimal;  that is, it has the minimal cost if and only if there
exists a $1$-Lipschitz real-valued  function $l$ on $M$ such that
\[l(x_i)-l(y_i)=d(x_i,y_i)\]
for all pairs $x_i,y_i$.
\end{theorem}

Denote by $\mathcal{L}$ the set of all $1$-Lipschitz functions on
$M$. The following is an immediate corollary of Theorem
\ref{T:WitnLip}.

\begin{corollary}
\begin{equation}\label{E:TCLip}\|f\|_{\tc}=\sup_{l\in \mathcal{L}} \left|\sum_{v\in M}
l(v) f(v)\right|.
\end{equation}
\end{corollary}

\begin{corollary} Embedding \eqref{E:EmbBasePt} is isometric
if we endow $\tp(M)$ with the norm $\|\cdot\|_{\tc}$.
\end{corollary}

It is not difficult to see that if $\mathcal{L}$ in the right-hand
side of \eqref{E:TCLip} is replaced with a subset $\mathcal{K}$ in
the set of all $1$-Lipschitz functions, one obtains a seminorm on
$\tp(M)$:
\begin{equation}\label{E:NormK}
\|f\|_{\mathcal{K}}=\sup_{l\in \mathcal{K}}|l(f)|,~\hbox{ where
}~l(f)=\sum_{v\in M} l(v) f(v).
\end{equation}
It is clear that embedding \eqref{E:EmbBasePt} is isometric as an
embedding from $M$ into the seminormed space $(\tp(M),
\|\cdot\|_{\mathcal{K}})$ if and only if for every two points
$u,v\in M$ and any $\ep>0$, there exists $l\in \mathcal{K}$ such
that $|l(u)-l(v)|\ge d(u,v)-\ep$. In this paper, our focus is
mainly on sets $\mathcal{K}$ satisfying the conditions:
\medskip

\noindent{\bf A.} All functions in $\mk$ are $1$-Lipschitz, that
is  ${\mathcal{K}}\subseteq {\mathcal{L}}.$
\medskip

\noindent{\bf B.} For every $u,v\in M$, there is a function
$l\in\mathcal{K}$ satisfying the condition $|l(u)-l(v)|=d(u,v)$.
\medskip

Obviously, for $\mk$ satisfying {\bf A} and {\bf B}, mapping
\eqref{E:EmbBasePt} is an isometric embedding of $M$ into
$(\tp(M), \|\cdot\|_{\mathcal{K}})$.
\medskip

The study of the seminormed spaces $(\tp(M),
\|\cdot\|_{\mathcal{K}})$ for $\mathcal{K}$ satisfying the
conditions {\bf A} and {\bf B} and different from $\mathcal{L}$
was initiated in \cite{MPV08}, and related results were obtained
in \cite{Zat08,Zat10}.
\medskip

Now we give two simple examples of function sets satisfying
conditions {\bf A} and {\bf B}.\smallskip

\noindent{\bf 1.} The set of all distance functions
$l_v(\cdot):=d(v,\cdot)$, $v\in M$. This set will be denoted by
$\mathcal{F}$ because  it was first used in the theory of metric
embeddings by Fr\'echet \cite{Fre10}.
\medskip

\noindent{\bf 2.} The set of all functions of the form
\begin{equation}\label{E:phi}\phi_{u,v}=\frac{d(v,\cdot)-d(u,\cdot)}2,\quad
u,v\in M.\end{equation} We denote this set by $\mathcal{DP}$ ({\it
double-point}). The set was introduced in \cite[Section
1.2.2]{MPV08}.
\medskip

In cases where $\|\cdot\|_\mk$ is a norm, we denote the completion
of the normed space $(\tp(M), \|\cdot\|_{\mathcal{K}})$ by
$\tp_\mk(M)$. Also, in the cases where $\mk=\ml, \mdp$, or $\mf$,
we use $\tp_\ml(M)$, $\tp_\mdp(M)$, or $\tp_\mf(M)$, respectively.
Observe that $\tp_\ml(M)=\tc(M)$. The same notation will be used
in cases where $\|\cdot\|_\mk$ is a seminorm. In such cases,
$\tp_\mk(M)$ denotes  the completion of the quotient of $\tp(M)$
over $\ker\|\cdot\|_\mk$ with respect to the norm induced by
$\|\cdot\|_\mk$ on this quotient.

\subsection{Statement of results}\label{S:Results}

Section \ref{S:Smallest} is devoted to analysis of the comment of
Melleray, Petrov, and  Vershik \cite[Comment~3, p.~185]{MPV08}
which can be, by Proposition \ref{P:AnyX}, restated as: {\it In
contrast to the existence of the maximal norm of the form
$\|\cdot\|_\mk$ with $\mk$ satisfying {\bf A} and {\bf B}, there
is no minimal norm of the form $\|\cdot\|_\mk$; moreover, it can
happen that for a given norm of this form the infimum of the norms
which are less than a given norm, is a seminorm, but not a norm.}

We show in Corollary \ref{C:NoMin}, that for finite metric spaces,
with exception of a small class, the infimum of seminorms of the
form $\|\cdot\|_\mk$ with $\mk$ satisfying {\bf A} and {\bf B} is
not a seminorm.

As it was discovered over the last decade, one of the differences
between  spaces $\tc(M)$  and other spaces of the form
$\tp_{\mk}(M)$, is a substantial ``presence'' of
$\ell_1$-subspaces in $\tc(M)$. Compare the results in
\cite{God10, Dal15, CD16, CDW16, CJ17, DKO18+} and Proposition
\ref{P:AnyX}, which shows that there does not have to be any such
presence in $\tp_{\mk}(M)$ for general $\mk$. Section
\ref{S:ell_1} is devoted to $\ell_1$-subspaces in $\tc(M)$. The
main result of this section, Theorem \ref{T:ell_1}, gives a
negative answer to the following question \cite[Question 2,
p.~3410]{CJ17}: {\it Let $M$ be an infinite uniformly discrete
metric space. Does $\tc(M)$ contain a subspace isometric to
$\ell_1$?} Recall that a metric space $M$ is called {\it uniformly
discrete} if there exists a constant $\delta
> 0$ such that
\[\forall u,v\in X~(u\ne v)\Rightarrow d_X(u,v)\ge\delta.\]

In Section \ref{S:KerDP} it is shown that there exist a class of
metric spaces for which the double-point norm introduced in
\cite{MPV08} is not a norm - it has a nontrivial kernel.

In Section \ref{S:DPforTree} it is proved that the space
$\tp_{\mdp}(M)$ for a tree $M$ is close to $\ell_\infty^n$ of the
corresponding dimension, and thus is quite different from the
transportation cost space on a tree.

Recall that a metric space is called {\it locally finite} if all
of its balls of finite radius have finite cardinality. In Section
\ref{S:Repr} a class of locally finite metric spaces $M$
satisfying the following condition is found: all Banach spaces
containing $M$ isometrically contain linear isometric copies of
$\ell_1$. In  particular, this is true for spaces of the form
$\tp_\mk(M)$.  This result is also motivated by the following
problem considered in \cite{KL08,OO19,OO19a}:

\begin{problem}\label{P:D(X)>1} For what Banach spaces $X$ do there exist locally finite
metric spaces $M$ such that each finite subset of $M$ embeds
isometrically into $X$, but $M$ does not embed isometrically into
$X$?
\end{problem}

Our result reveals a new  class of Banach spaces for which the
phenomenon described in Problem \ref{P:D(X)>1} occurs.

\subsection{Some interesting directions in the theory of $\tp_\mk(M)$
spaces}\label{S:Directions}

We refer to \cite{Mau03} for basic theory of cotype of Banach
spaces, and to \cite{Ost13} for relevance of this theory for
metric embeddings as well as for an additional background needed
for reading this section.\medskip

In our opinion, one of the most interesting and challenging
directions in the study of the spaces $\tp_\mk(M)$ is related to
the following well-known facts: (a) Finite metric spaces admit
bilipschitz embeddings with distortions arbitrary close to $1$
into every Banach space with trivial cotype; (b) Locally finite
metric spaces admit bilipschitz embeddings into every Banach space
with trivial cotype, whose distortions are bounded by an absolute
constant, see \cite{BL08,Ost12} and \cite[Chapters 1 and
2]{Ost13}; in \cite{OO19} it was shown that this constant does not
exceed $4+\ep$ for every $\ep>0$.

In order to determine whether  isometric embeddings of $M$ into
$\tp_\mk(M)$ are of a different nature, it is crucial to develop
tools needed to advance the following direction of research.

\begin{dir}\label{D:NL} Characterize metric spaces $M$ for which we can find a set $\mk$ of functions on $M$ satisfying conditions {\bf A} and {\bf B},
and such that $\tp_\mk(M)$ has nontrivial cotype.
\end{dir}

The discussion presented in \cite[Section 11.1]{Ost13} suggests
the important relevant problem:

\begin{problem}\label{P:ExpNL}
Can one find a sequence $\{G_n\}_{n=1}^\infty$ of expanders  and
the corresponding sets $\mathcal{K}_n$ of Lipschitz functions
satisfying conditions {\bf A} and {\bf B} such that the direct sum
$\left(\oplus_{n=1}^\infty \tp_{\mk_n}(G_n)\right)_2$ has
nontrivial cotype?
\end{problem}

It should be mentioned that, by virtue of  Proposition
\ref{P:AnyX},  each Banach space containing $M$ isometrically also
contains a subspace isometric to $\tp_\mk(M)$ for suitably chosen
$\mk$ satisfying {\bf A} and {\bf B}. This shows the significance
of the following general direction in the study of $\tp_\mk(M)$:

\begin{dir} Given a metric space $M$, find sets $\mk$ of
functions on $M$ satisfying conditions {\bf A} and {\bf B} for
which one can describe the Banach-space-theoretical structure of
$\tp_\mk(M)$.
\end{dir}

Another direction which we regard as fruitful is:

\begin{dir}\label{D:AllTP_K} Find metric spaces $M$ for which
the linear structure of Banach spaces $\tp_\mk(M)$ satisfies
certain geometric conditions for every choice of $\mk$ satisfying
the conditions {\bf A} and {\bf B}.
\end{dir}

Proposition \ref{P:AnyX} reveals that Direction \ref{D:AllTP_K} is
similar to the following: {\it Find metric spaces for which the
linear structure of Banach spaces admitting an isometric embedding
of $M$ satisfies certain geometric conditions.}

A few results of this type are already available: Godefroy and
Kalton \cite{GK03} proved that if $M$ is a separable Banach space,
then any Banach space containing isometric copy of $M$ contains a
linearly isometric copy of $M$. Dutrieux and Lancien \cite{DL08}
introduced the notion of a {\it representing subset} and found
several interesting examples of such sets. The definition of
representing subsets is provided in Section \ref{S:Repr}, where
the existence of locally finite metric spaces representing
$\ell_1$ is proved.

\subsection{Isometric embeddings into Banach spaces and spaces
$\tp_{\mk}(M)$}

The goal of this section is to show that Banach spaces of the form
$\tp_\mk(M)$ are present in all Banach spaces containing isometric
copies of $M$.

\begin{proposition}\label{P:AnyX} If a metric space $M$ admits an isometric embedding into a
Banach space $X$, then there exists a set $\mk$ of functions on
$M$ satisfying {\bf A} and {\bf B} such that $\tp_\mk(M)$ is
linearly isometric to a subspace of $X$.
\end{proposition}

\begin{remark} Since $\|\cdot\|_\mk$ can have a nontrivial kernel,
a linear isometry $E:\tp_\mk(M)\to X$ is understood as a linear
map satisfying $\|Ex\|_X=\|x\|_\mk$, although this map can have a
kernel.
\end{remark}

\begin{proof} The isometric image of $M$ in $X$ may be shifted so that one of the elements of $M$ coincides with $0$. With this in mind, let us
identify elements of $M$ and their isometric images in $X$.

Denote by $X^*$ the dual space of $X$ and by $S(X^*)$ its unit
sphere. It is clear that the restrictions of elements of $S(X^*)$
to $M$ are $1$-Lipschitz functions. Denote by $\mathcal{S}$ this
set of restrictions. Let us show that the space $\tp_\ms(M)$
admits a linear isometric embedding into $X$ given by
$\1_v-\1_0\mapsto v$, where $v\in M$ is identified with its image
in $X$.

It suffices to establish that, for any finite collections
$\{a_i\}\subset\mathbb{R}$ and $\{v_i\}\subset M$, the equality
\[\left\|\sum_i
a_i(\1_{v_i}-\1_0)\right\|_\ms=\left\|\sum_ia_iv_i\right\|_X\]
holds. Setting $f=\sum_ia_i(\1_{v_i}-\1_0)\in\tp(M)$, one arrives
at:

\[\begin{split}\left\|\sum_ia_iv_i\right\|_X&=\sup_{x^*\in
S(X^*)}\left|\sum_ia_ix^*(v_i)\right|\\&= \sup_{x^*\in
S(X^*)}\left|\sum_ia_i(x^*(v_i)-x^*(0))\right|\\&=
\sup_{l\in\ms}\left|\sum_{v\in M}l(v)f(v)\right|\\&=\left\|\sum_i
a_i(\1_{v_i}-\1_0)\right\|_\ms.
 \qedhere\end{split}\]
\end{proof}

\begin{remark} In particular, $\tp_\mk(M)$ can be strictly convex,
which makes it different from $\tp_\ml(M)$, see \cite[Proposition
2]{CJ17}. It should be  mentioned that some metric spaces, e.g.
unweighted graphs which are neither complete graphs nor paths
\cite[Observation 5.1]{Ost13b} do not admit isometric embeddings
into strictly convex Banach spaces.
\end{remark}

\subsection{Historical and terminological
remarks}\label{S:HistTerm}

We are aware of three directions of research for which it was
natural to introduce notions which either coincide or are closely
related to the notions of the transportation cost and the
transportation cost space. These are:

\begin{itemize}

\item[(1)] Study of algebraically ``free'' topological (or metric)
structures which contain a given topological (or metric) structure
as a substructure.

\item[(2)] Developing the notion of a distance between two
probability distributions on a metric space.

\item[(3)] Studying the notion of a transport of one finite
positive measure into another.

\end{itemize}

Some of the works representing direction (1) are: Markov
\cite{Mar41, Mar45}, Shimrat \cite{Shi54}, Arens-Eells
\cite{AE56}, Michael \cite{Mic64}, Kadets \cite{Kad85}, Pestov
\cite{Pes86}, Weaver \cite{Wea99}, Godefroy-Kalton \cite{GK03}.

Arens and Eells \cite{AE56} introduced, for a metric space $M$,
the linear space which we denote $\tp(M)$ and the norm on it,
which we denote $\|\cdot\|_\tc$. Their goal for introducing these
notions was to prove the following result: ``Every metric space
can be isometrically embedded as a closed subset of a normed
linear space''. In this connection, they derived a version of
Theorem \ref{T:WitnLip}, more precisely, they proved that the dual
of the normed space $(\tp(M),\|\cdot\|_\tc)$ is the space of
Lipschitz functions vanishing at a base point. Arens and Eells
\cite{AE56} did not consider the completion of this normed space
because for the completion the stated above result is false.

The completion of the space constructed by Arens and Eells was
considered by Kadets \cite{Kad85}, Pestov \cite{Pes86}, Weaver
\cite{Wea99}, and Godefroy-Kalton \cite{GK03} (see also
\cite{Mic64}). Weaver \cite[Definition 2.2.1]{Wea99} defined, what
he named {\it Arens-Eells space}, as the completion of
$(\tp(M),\|\cdot\|_\tc)$. Kadets, Pestov, and Godefroy-Kalton
defined an equivalent object in the dual way. Namely, they
considered the space $\lip_0(M)$ of all Lipschitz functions on the
space $M$ which vanish at a base point $O$. This is a Banach space
with respect to  the norm defined as the Lipschitz constant.
Kadets, Pestov, and Godefroy-Kalton consider the closed subspace
of $(\lip_0(M))^*$ spanned by the point evaluation functionals on
$\lip_0(M)$. We denote the point evaluation functional
corresponding to point $x$ by $\delta(x)$ and exclude $\delta(O)$
from consideration because it is a zero functional.

\begin{observation}\label{O:LFvsTC} The norm of a finite linear combination $\sum_{x\in
A}a_x\delta(x)$ in the dual space $(\lip_0(M))^*$  is the same as
the transportation cost of the transportation problem $\sum_{x\in
A}a_x(\1_x-\1_O)$.
\end{observation}

This observation follows immediately from the fact that
$\lip_0(M)$ is the dual of $(\tp(M),\|\cdot\|_\tc)$. Thus, the
spaces studied by Kadets, Pestov, and Godefroy-Kalton are all
isometric to the completion of $(\tp(M),\|\cdot\|_\tc)$. Kadets
\cite{Kad85} denoted this space $\widetilde X$ and did not give
any name to it, Pestov \cite{Pes86} called it the {\it free Banach
space}, and Godefroy-Kalton \cite{GK03} called it the {\it
Lipschitz-free Banach space}. Thus all these names correspond to
spaces which are canonically (in the sense of Observation
\ref{O:LFvsTC}) isometric to the space $\tc(M)$.
\medskip

Apparently, it is impossible to list all works corresponding to
the directions (2) and (3). An ample bibliography, a wide range of
contributors, and relevant discussions are presented by Villani
\cite[pp.~106--111]{Vil09}. We mention only the names of
Kantorovich \cite{Kan42,Kan11}, Kantorovich-Gavurin \cite{KG49},
and Kantorovich-Rubinstein \cite{KR57, KR58} for direction (3) and
Vasershtein \cite{Vas69} (currently spelled as Wasserstein) for
direction (2).
\medskip

Kantorovich and Gavurin \cite{KG49} introduced $\tp(M)$,
$\|\cdot\|_\tc$, observed Theorem \ref{T:WitnLip}, and developed
an approach to finding transportation plans of minimum cost.

A nice source for learning the basic definitions and results of
directions corresponding to (2) and (3) above is \cite[Chapters 1
and 7]{Vil03}. Another interesting source is \cite[\S 4 in Chapter
VIII]{KA84} (see also English translation in \cite{KA82}). Note
that in \cite{Vil03} the discussion is mostly limited to
probability measures, while in \cite{KA82,KA84} the discussion is
limited to compact metric spaces. As it is pointed out in
\cite{Vil03}, a  passage from arbitrary finite positive Borel
measures to probability measures can be achieved by normalization.
As for compactness, all of the main results can be generalized to
the setting of general complete separable metric space (see
\cite{Vil03}), such spaces are also called {\it Polish spaces}.
There are some obstacles which are to be overcame for some of more
general spaces, see Remark \ref{R:Gen}.

Let us present basic notions of the theory developed in
\cite{KA82,KA84,Vil03}. For a Polish space $(M,d)$, let $\Bs(M)$
denote the linear space of all finite Borel probability measures
$\mu$ on $X$ satisfying
\begin{equation}\label{E:BddMom}\int_M d(x,x_0)\,d\mu(x)<\infty\end{equation} for some (hence all)
$x_0\in M$. A {\it coupling} of a pair of finite positive Borel
measures $(\mu,\nu)$ with the same total mass on $M$ is a Borel
measure $\pi$ on $M\times M$ such that $\mu(A)=\pi(A\times M)$ and
$\nu(A)=\pi(M\times A)$ for every Borel measurable $A\subset M$.
The set of couplings of $(\mu,\nu)$ is denoted $\Pi(\mu,\nu)$. The
quantity
$$
\mathcal{T}_1(\mu,\nu):= \inf_{\pi\in
\Pi(\mu,\nu)}\bigg(\iint_{M\times M} d(x,y)\,d\pi(x,y)\bigg)
$$
is called the {\it minimal translocation work} between $\mu,\nu\in
\Bs(X)$ in \cite{Kan42} and the {\it optimal transportation cost}
in \cite[p.~3]{Vil03}. Kantorovich and Rubinstein \cite{KR57,KR58}
(see also \cite[Section 7.1]{Vil03}) proved that the set
$\mathcal{M}$ of all differences $\mu-\nu$ of measures satisfying
the conditions above forms a normed space if we endow it with the
norm
\begin{equation}\label{E:KR}\|\mu-\nu\|_{\kr}=\mathcal{T}_1(\mu,\nu),\end{equation} and that the
dual of the space $(\mathcal{M},\|\cdot\|_\kr)$ is the space
$\lip_0(M)$ with its usual norm. Observe that in the case where
$\mu$ and $\nu$ are atomic measures with finitely many atoms, the
difference $\mu-\nu$ can be regarded as an element of $\tp(M)$,
and we have $\|\mu-\nu\|_{\kr}=\|\mu-\nu\|_{\tc}$.

The normed space $(\mathcal{M}, \|\cdot\|_\kr)$ is not complete if
the space $M$ is not uniformly discrete. In fact, if there is a
sequence of pairs $x_i,y_i\in M$ such that all elements of the set
$\{x_i,y_i\}$ are distinct and $d(x_i,y_i)\le 2^{-i}$, then the
sequence of measures
$\left\{\sum_{i=1}^n(\delta(x_i)-\delta(y_i))\right\}_{n=1}^\infty$,
where $\delta(x)$ is the unit atomic measure supported on $\{x\}$,
converges in the norm described above, but not to a difference of
two finite measures. This example is well-known, see
\cite[Proposition 2.3.2]{Wea99}.
\medskip

The relation between the normed spaces $(\mathcal{M},
\|\cdot\|_\kr)$  and the transportation cost spaces (as we define
them in Section \ref{S:Def}) is described in the following result
(see Weaver \cite[Section 2.3]{Wea99} or \cite[Section 3.3]{Wea18}
for the case where the metric space $M$ is compact), which shows
that the completion of $(\mathcal{M},\|\cdot\|_\kr)$ coincides
with $\tc(M)$. We believe that this result is known to experts.
However, since a suitable reference has not been found, its proof
is presented below.

\begin{theorem}\label{T:TCvsKR} If $(M,d)$ is a Polish metric space, the space
$(\tp(M),\|\cdot\|_\tc)$ is dense in
$(\mathcal{M},\|\cdot\|_\kr)$. Hence $\tc(M)$ can be regarded as
the completion of $(\mathcal{M},\|\cdot\|_\kr)$.
\end{theorem}

\begin{proof} Fix the base point $O$. Since $\tp(M)\subset \mathcal{M}$ and for both spaces the dual space is $\lip_0(M)$, it suffices to show
that for each pair $\mu,\nu$ of finite Borel measures on $M$
satisfying \eqref{E:BddMom} and having the same total masses, and
for each $\ep>0$, there exists $f\in\tp(M)$ such that, for every
$1$-Lipschitz function $l$ on $M$ satisfying $l(O)=0$, there
holds:

\begin{equation}\label{E:Des}|l(\mu-\nu)- l(f)|<\ep.\end{equation}

Denote by $B(O,R)$ the closed ball in $M$ of radius $R$ centered
at $O$. Using condition \eqref{E:BddMom} for both $\mu$ and $\nu$,
we conclude that there exists $R\in (0,\infty)$ satisfying

\[\left|\int_{M\backslash B(O,R)}l(v) d\mu(v)\right|\le \int_{M\backslash B(O,R)}d(O,v) d\mu(v)<\frac{\ep}6\]
and
\[\left|\int_{M\backslash B(O,R)}l(v) d\nu(v)\right|\le \int_{M\backslash B(O,R)}d(O,v) d\nu(v)<\frac{\ep}6.\]

By Ulam's theorem (see \cite[Theorem 1.4]{Bil68} and \cite[Theorem
16.3.1]{Gar18}), there exists a compact set $K\subset B(O,R)$ such
that
\[\mu(B(O,R)\backslash K)<\frac{\ep}{6R}~\hbox{ and }~
\nu(B(O,R)\backslash K)<\frac{\ep}{6R},\] whence both
\[\left|\int_{B(O,R)\backslash K}l(v) d\mu(v)\right|<\frac{\ep}6\]
and
\[\left|\int_{B(O,R)\backslash K}l(v) d\nu(v)\right|<\frac{\ep}6.\]

Next, we split $K$ into a finite number $\{K_n\}_{n\in T}$ of
pairwise disjoint Borel subsets of diameter
$<\ep/(3(\mu(M)+\nu(M)))$ each. Define a function $f$ on $M$ as
follows: In each of the sets $K_n$, we pick a point $t_n$ and let
$f(t_n)=\mu(K_n)-\nu(K_n)$ for $n\in T$. If $f$ is extended as $0$
to the rest of $M$, the obtained function is not necessarily in
$\tp(M)$. To balance this, the $0$-extension is modified at point
$O$ by taking $f(O)=-\sum_{n\in T}f(t_n)$, thus implying
$f\in\tp(M)$. If $t_n\ne O$, the following inequality holds:
\[\begin{split}&\left|\int_{K_n}l(v)d(\mu-\nu)(v)-l(t_n)f(t_n)\right|=
\left|\int_{K_n}(l(v)-l(t_n))d(\mu-\nu)(v)\right|\\&\qquad\le
\int_{K_n}|l(v)-l(t_n)|d\mu(v)+\int_{K_n}|l(v)-l(t_n)|d\nu(v)\le
\frac\ep3\cdot\frac{\mu(K_n)+\nu(K_n)}{\mu(K)+\nu(K)},\end{split}\]
where, in the last inequality,  the fact that $l$ is $1$-Lipschitz
and the assumption on the diameter of $K_n$ are used.

Since $l(O)=0$, one has:
\[\begin{split}|l&(\mu-\nu)- l(f)|\le \sum_{n\in
T}\left|\int_{K_n}l(v)d(\mu-\nu)(v)-l(t_n)f(t_n)\right|+\left|\int_{B(O,R)\backslash
K}l(v) d\mu(v)\right|\\&+\left|\int_{B(O,R)\backslash K}l(v)
d\nu(v)\right|+\left|\int_{M\backslash B(O,R)}l(v)
d\mu(v)\right|+\left|\int_{M\backslash B(O,R)}l(v)
d\nu(v)\right|\\&<\frac\ep3+\frac\ep6+\frac\ep6+\frac\ep6+\frac\ep6=\ep.\qedhere\end{split}
\]
\end{proof}

\begin{remark}\label{R:Gen} We do not know how to prove an analogue of Theorem \ref{T:TCvsKR} in the case of general metric
spaces. See \cite[p.~234--235]{Bil68} in this connection.
Nonseparability will not be an obstacle if we consider uniformly
discrete spaces, because in such spaces finite Borel measures have
countable support. The general metric case is beyond the subject
of this work.  For a systematic exposition of generalizations of
results of Kantorovich and Rubinstein we refer the reader to
\cite{RR98}.
\end{remark}

Finally, it is worth  mentioning that in Computer Science the
transportation cost is often called {\it Earth Mover's Distance}.
This name was introduced by Rubner-Tomasi-Guibas \cite{RTG98} in
their work on computer vision. They knew the notion of
transportation cost and their goal was to generalize it to the
cases where total demand can be less than the total supply. This
more general case is not considered in the present paper as such
generalized transportation problems do not form a vector space.
\medskip

Several authors published their opinions on the most suitable
choice of the name for the spaces which we call {\it
transportation cost spaces}. Vershik \cite{Ver04,Ver13} provided
an argument in favor of the name {\it Kantorovich space}. Villani
\cite[pp.~106--107]{Vil09} decided in favor of Wasserstein, and
Weaver \cite[p.~125]{Wea18} defended the name {\it Arens-Eells
space}. It is interesting to mention that Villani decided in favor
of Wasserstein only because this term is more popular on the
Internet than the others. In view of the information presented
above, the argument of Vershik towards the {\it Kantorovich space}
is the most convincing. However, we decided not to follow
Vershik's suggestion because the term {\it Kantorovich space} has
already become a standard term for another object in Functional
Analysis, see, for example, \cite{KK04}. On the other hand, the
relevance of the term {\it transportation cost} is mentioned even
by the authors who are on the side of other terms, see
\cite[p.~762]{Nao18}, \cite[Introduction]{Vil03} ({\it cost of
transfer plan} in \cite[Section 3.3]{Wea18}).

\section{On the smallest seminorm of the type $\|\cdot
\|_{\mathcal{K}}$}\label{S:Smallest}

Obviously,  for every metric space $M$, the norm
$\|\cdot\|_{\mathcal{L}}$ is the largest seminorm of the type
$\|\cdot \|_{\mathcal{K}}$ for $\mathcal{K}$ satisfying conditions
{\bf A} and {\bf B}. In contrast, the class of finite metric
spaces $M$ for which $\tp(M)$ possesses the smallest seminorm of
the type $\|\cdot \|_{\mathcal{K}}$ for $\mathcal{K}$ satisfying
conditions {\bf A} and {\bf B} is rather narrow; and the goal of
this section is to present its complete description. We start with
a very simple existence result.

\begin{proposition}\label{P:SubR} Let $M$ be a finite subset of $\mathbb{R}$
with the induced metric. Then there exists the smallest seminorm
of the form $\|\cdot \|_{\mathcal{K}}$ on $\tp(M)$.
\end{proposition}

\begin{proof} In fact, let $M=\{x_1,\dots,x_n\}\subset\mathbb{R}$ with
$x_1<x_2<\dots<x_n$. By the condition {\bf B}, the set
$\mathcal{K}$ should contain a $1$-Lipschitz function $l$ such
that $|l(x_1)-l(x_n)|=|x_1-x_n|$.  It is clear that, for the same
function $l$, we have $|l(x_i)-l(x_j)|=|x_i-x_j|$, and also that
any two such functions can by obtained from each other by adding a
constant and multiplying by $\pm 1$. Thus, any $1$-element set
$\mathcal{K}$ containing such function leads to the minimal
semi-norm of the described type.\end{proof}

In the sequel, the following generalization of the notion of a
linear triple (\cite[p.~56]{Blu53}) will be used.

\begin{definition}\label{D:LinearnTup} A collection $r=\{r_i\}_{i=1}^n$, $n\ge 3$, of points in a metric space $(M,d)$ is called a {\it linear
tuple} if the sequence $\{d(r_i,r_1)\}_{i=1}^n$ is strictly
increasing and if, for $1\le i<j<k\le n$, the equality below
holds:
\begin{equation}\label{E:TriangleEq}
d(r_i,r_k)=d(r_i,r_j)+d(r_j,r_k).
\end{equation}
A {\it linear triple} is a linear tuple with $n=3$.
\end{definition}

Our next goal is to describe the condition which, for finite
metric spaces $M$, is equivalent to existence of the smallest
seminorm of the form  $\|\cdot\|_\mk$ on $\tp(M)$.
\medskip

In the following we assume that a metric space $M$ contains at
least two points.

\begin{definition} We say that a metric space $M$ satisfies the {\it min-condition} if it contains a finite set of pairs
$\{u_i,v_i\}$ $(i\in \mathcal{I})$ having the following two
properties:

\begin{itemize}

\item[{\bf I.}] For each $i\in \mathcal{I}$, every  point $x$ in
$M$, different from $u_i$ and $v_i$, is such that $u_i,x,v_i$ is a
linear triple.

\item[{\bf II.}] For each pair $\{x,y\}$ of distinct points in $M$
there is $i\in \mathcal{I}$ such that exactly one of the following
four conditions holds:

\begin{enumerate}

\item The pairs  $\{x,y\}$ and $\{u_i,v_i\}$ coincide.

\item  Exactly one of the points  $\{x,y\}$ coincides with one of
the points $\{u_i,v_i\}$, and the remaining point, denote it $z$,
is such that $u_i,z,v_i$ is a linear triple.

\item $u_i,x,y,v_i$ is a linear tuple.

\item $u_i,y,x,v_i$ is a linear tuple.

\end{enumerate}

If pairs $\{u_i,v_i\}$ and $\{x,y\}$ in $M$ satisfy one of the
conditions 1--4 in {\bf II} we say that $\{x,y\}$ is {\it on a
geodesic between} $u_i$ and $v_i$.

\end{itemize}

\end{definition}

To exemplify this definition,  notice  that a finite subset $M$ of
$\mathbb{R}$ satisfies the min-condition and the corresponding set
of pairs consists of one pair $\{u_1,v_1\}$ where $u_1$ is the
minimal element of $M$ and $v_1$ is the maximal element of $M$.

More interesting examples of metric spaces satisfying the
min-condition are even cycles - in terminology of Graph Theory -
with their graph distances. More general examples are weighted
even cycles provided that the weights are symmetric in the
following sense: if we label vertices by $x_1,\dots,x_{2n}$ in the
cyclic order, the weight of the edge joining $x_k$ and $x_{k+1}$
is the same as the weight of the edge joining $x_{n+k}$ and
$x_{n+k+1}$, where the addition is $\mod(2n)$. The corresponding
set of pairs is the set $\{x_i,x_{i+n}\}$, $i=1,\dots,n$.

The main result in this section is:

\begin{theorem}\label{T:DNE} A finite metric space $M$ satisfies the min-condition if and only if there exists the smallest seminorm of the
form $\|\cdot\|_\mk$ on $\tp(M)$ with $\mk$ satisfying conditions
{\bf A} and {\bf B}.
\end{theorem}

\begin{proof}  ``Only if'': For each $i$, consider a $1$-Lipschitz function $l_i$ such that $d(u_i,v_i)=|l_i(u_i)-l_i(v_i)|$. Condition {\bf I}
implies that the function $l_i$ is uniquely determined up to
addition of a constant and multiplication by $-1$. Therefore, each
set ${\mathcal{K}}$ satisfying both {\bf A} and {\bf B} should
contain at least one representative from each set $\mathbb{L}_i$
of functions obtained from $l_i$ by adding all possible real
constants and multiplying by $\pm 1$. Let $\mr$ be a collection of
such representatives, we select one representative in each
$\mathbb{L}_i$. Evidently, $\|\cdot\|_\mk\ge \|\cdot\|_\mr$. On
the other hand, condition {\bf II} implies that any such $\mr$
satisfies condition {\bf B}. Thus, the seminorm corresponding to
the set $\mr$ is the smallest seminorm of the desired type.
\medskip

``If'': Let us construct a set which can be called a {\it minimal
set of pairs} as follows. Starting with the set comprising all
pairs of distinct points in $M$, we remove those pairs which are
on geodesics between other pairs. This procedure results in  a set
of pairs $\{(u_i,v_i)\}$ satisfying {\bf II}. If it satisfies {\bf
I}, then $M$ satisfies the min-condition.

It is clear that to complete the proof it suffices to show that if
the obtained set of pairs does not satisfy {\bf I}, then for any
set $\mk$ satisfying {\bf A} and {\bf B} on $M$, there is another
set $\widetilde\mk$ satisfying {\bf A} and {\bf B} on $M$ and a
function $f\in\tp(M)$ such that $\|f\|_{\widetilde\mk}<\|f\|_\mk$.

It can be noticed  that $\mk$ is equivalent - in the sense that
induces the same seminorm $\|\cdot\|_\mk$ -  to some set of
$1$-Lipschitz functions containing $1$-Lipschitz functions $l_i$
satisfying $l_i(u_i)=0$, $l_i(v_i)=d(u_i,v_i)$.  It may be assumed
without loss of generality that the pair $(u_1,v_1)$ does not
satisfy {\bf I}, implying that there is $w\in M$ such that

\begin{equation}\label{E:(1)}
d(u_1,v_1)<d(u_1,w)+d(w,v_1).
\end{equation}

In addition, the next two inequalities hold because otherwise the
pair $(u_1,v_1)$ should be deleted from the minimal collection of
pairs:

\begin{equation}\label{E:(2)}
d(u_1,w)<d(u_1,v_1)+d(v_1,w)
\end{equation}

\begin{equation}\label{E:(3)}
d(v_1,w)<d(u_1,v_1)+d(u_1,w)
\end{equation}

The value of $l_1(w)$ is in the interval $[d(u_1,v_1)-d(w,v_1),
d(u_1,w)]$. On the other hand,  inequality \eqref{E:(1)} implies
that this interval does not reduce to one point, and consequently
at least one of the following inequalities holds:
\begin{equation}\label{E:Cond's} l_1(w)< d(u_1,w),\quad
l_1(w)>d(u_1,v_1)-d(w,v_1).\end{equation}

Despite the asymmetry between the different conditions in
\eqref{E:Cond's} caused by different roles of $u_1$ and $v_1$ in
the definition of $l_1$, one can check that the cases in
\eqref{E:Cond's} can be considered in a similar way. Thence, it
suffices only to consider the case $l_1(w)< d(u_1,w)$. In this
case,  consider the function $f\in\tp(M)$ given by
\[f=\tau(\1_{v_1}-\1_{u_1})+(\1_{v_1}-\1_w),\]
where $\tau>0$ will be selected later. Let $\widetilde \mk$ be
given as the set of all functions of the form $l_z:=d(z,\cdot)$
where $z$ is any of the endpoints of pairs $\{(u_i,v_i)\}$ except
$v_1$. The choice of $\{(u_i,v_i)\}$ implies that $\widetilde \mk$
satisfies condition {\bf B}. It is clear that $\widetilde \mk$
also satisfies condition {\bf A}. It remains to show that there
exists $f\in\tp(M)$ such that $\|f\|_\mk>\|f\|_{\widetilde \mk}$.

Observe that
\[\|f\|_\mk\ge |l_1(f)|=\left|(\tau+1)d(u_1,v_1)-l_1(w)\right|.\]
We assume that $\tau>0$ is large enough to ensure that
\[|l_1(f)|=(\tau+1)d(u_1,v_1)-l_1(w)>(\tau+1)d(u_1,v_1)-d(u_1,w)>0,\]
whence $l_1(f)>l_{u_1}(f)>0$.  To complete the proof of
$\|f\|_\mk>\|f\|_{\widetilde \mk}$, it has to be shown that for a
suitably chosen $\tau>0$, one has $l_1(f)>|l_z(f)|$ for all $z$ of
the described type. Indeed,
\[l_z(f)=\tau(d(z,v_1)-d(z,u_1))+(d(z,v_1)-d(z,w)),\]
and, therefore
\[|l_z(f)|\le\tau|d(z,v_1)-d(z,u_1)|+d(v_1,w).\]

To achieve the desired goal, observe that
\[|d(z,v_1)-d(z,u_1)|<d(u_1,v_1)\]
for every $z\ne u_1$, because otherwise $\{u_1,v_1\}$ would be on
a geodesic of either between $u_1$ and $z$, or between $v_1$ and
$z$. In any of the cases we get a contradiction with the fact that
the pair $\{u_1,v_1\}$ belongs to the minimal set of pairs. Thus,
for a sufficiently large $\tau>0$, the inequality
\[(\tau+1)d(u_1,v_1)-d(u_1,w)>\tau|d(z,v_1)-d(z,u_1)|+d(v_1,w)\]
holds for all $z\ne u_1$ belonging to the described set, and we
are done.
\end{proof}

\begin{corollary}\label{C:NoMin} Let $M$ be a finite metric space.
If $M$ does not satisfy the min-condition, then the infimum of all
seminorms $\|\cdot\|_\mk$ on $\tp(M)$ over $\mk$ satisfying the
conditions {\bf A} and {\bf B} is not a seminorm.
\end{corollary}

\begin{proof} Assume the contrary. Let $\|\cdot\|_{\inf}$
be the seminorm on $\tp(M)$, which is an infimum of all seminorms
of the form $\|\cdot\|_\mk$, with $\mk$ satisfying conditions {\bf
A} and {\bf B}. Let $X$ be the quotient of the seminormed space
$(\tp(M),\|\cdot\|_{\inf})$ by the kernel of $\|\cdot\|_{\inf}$
and
 $O$ denote the base point in $M$. For each element $v \in M$,
denote by $\tilde v$ the image of $\1_v-\1_O$ in $X$. It follows
from the description of conditions {\bf A} and {\bf B} - see the
paragraph below the description - that the map $v\mapsto \tilde v$
is an isometry of $M$ into $X$. It is also clear that $\tilde
O=0$.

By Proposition \ref{P:AnyX}, there exists a seminorm of the form
$\|\cdot\|_\mn$ on $\tp(M)$ with $\mn$ satisfying {\bf A} and {\bf
B} and such that
\[\left\|\sum_ia_i(\1_{v_i}-\1_{O})\right\|_\mn=\left\|\sum_ia_i{\tilde v_i}\right\|_X\]
for any $\sum_ia_i(\1_{v_i}-\1_{O})\in\tp(M)$. On the other hand,
by the construction \[\left\|\sum_ia_i{\tilde
v_i}\right\|_X=\inf_\mk\left\|\sum_ia_i(\1_{v_i}-\1_{O})\right\|_\mk.\]

Since $M$ does not satisfy the min-condition, by Theorem
\ref{T:DNE}, $\|f\|_{\inf}$ is strictly less than $\|f\|_\mn$ for
some functions $f\in\tp(M)$, which is a contradiction because any
function in $\tp(M)$ can be written in the form
$\sum_ia_i(\1_{v_i}-\1_{O})$.
\end{proof}

\begin{remark} For small metric spaces the result of Corollary
\ref{C:NoMin} admits a simple direct proof. Consider, for example,
an equilateral set $M=\{a,b,c\}$ with all distances equal to $1$.
Let $\mk_1=\{d(x,a), d(x,b)\}$ and $\mk_2=\{d(x,b),
d(x,c)\}$. 
It is clear that both sets satisfy the conditions {\bf A} and {\bf
B}. It is easy to check that
\[\|2\1_a-\1_b-\1_c\|_{\mk_2}=1 \quad\hbox{ and }\quad
\|\1_a+\1_b-2\1_c\|_{\mk_1}=1.\] Therefore
\[\|2\1_a-\1_b-\1_c\|_{\inf}\le1 \quad\hbox{ and }\quad
\|\1_a+\1_b-2\1_c\|_{\inf}\le 1.\] On the other hand
\[\|(2\1_a-\1_b-\1_c)+(\1_a+\1_b-2\1_c)\|_{\inf}=\|3(\1_a-\1_c)\|_{\inf}.\]
If $\|\cdot\|_{\inf}$ would be a seminorm, this would lead to a
contradiction because $\|3(\1_a-\1_c)\|_{\mk}=3$ for each $\mk$
satisfying {\bf A} and {\bf B}.
\end{remark}

\section{On $\ell_1$-subspaces in $\tp_{\ml}(M)$}\label{S:ell_1}

The next statement brings out one of the main outcomes of this
paper.

\begin{theorem}\label{T:ell_1} There exists an infinite uniformly discrete metric space $M$ such that
$\tc(M)$ does not contain an isometric copy of $\ell_1$.
\end{theorem}

\begin{proof} Consider a metric space whose vertex set is $\mathbb{N}$, while its metric is quite different from the
standard. It is a close-to-equilateral metric defined as follows.
Let $h:\mathbb{N}\to (1,2)$ be a strictly increasing function and
the metric $d$ be given by

\begin{equation}\label{E:Def_d} d(i,j)=\begin{cases} h(\min\{i,~j\}) &\hbox{ if }i\ne j\\
0 &\hbox{ if }i=j.\end{cases}\end{equation} It is clear that $d$
is a metric for any choice of $h$. This metric space is a
generalization of the space suggested in \cite[Remark~10,
Example~2]{CJ17}.
\medskip

In this case, one can find a very handy description of the space
$\tc(M)$ and the norm on it, which turns out to be equivalent to
the $\ell_1$-norm $\|f\|_1$, which is well defined since $f$ is a
real-valued function on $\mathbb{N}$. In fact, let $f\in\tp(M)$,
then the amount of the available product is equal to $\|f\|_1/2$.
Since each unit of product is to be moved to a distance which is
between $1$ and $2$ (see \eqref{E:Def_d}), we get that the cost of
an optimal transportation plan is between $\|f\|_1/2$ and
$\|f\|_1$. Observe also that $\tp(M)$ contains all finitely
supported sequences contained in the kernel of the functional
$(1,\dots,1,\dots)\in\ell_\infty$. Therefore the space $\tc(M)$
consists of all sequences of the intersection
$\ell_1\cap\ker(1,\dots,1,\dots)$, and its norm satisfies
$\|f\|_1/2\le \|f\|_\tc\le \|f\|_1$.

 For the sequel, it will be convenient to introduce
the notion of a {\it (generalized) transportation plan} for
$f\in\tc(M)$ as a representation of $f$ by means of a convergent
series of the form:
\begin{equation}\label{E:InfTranPlan} f=\sum_{i=1}^\infty a_i(\1_{x_i}-\1_{y_i})\end{equation} with $a_i>0$ and
\begin{equation}\label{E:Sum_ai}\sum_{i=1}^\infty
a_i<\infty.\end{equation}

 We would like to emphasize that condition
\eqref{E:Sum_ai} is different from the condition in the standard
description of the completion, which in our case can be described
as the set of sums of all series of the form
\[\sum_{k=1}^\infty\sum_{i=s_k+1}^{s_{k+1}}a_i(\1_{x_i}-\1_{y_i})\]
for some $0=s_1<s_2<\dots<s_k<\dots$, $\{x_i\}$, $\{y_i\}$, and
$\{a_i\}$ with
\begin{equation}\label{E:Complet}\sum_{k=1}^\infty\left\|\sum_{i=s_k+1}^{s_{k+1}}a_i(\1_{x_i}-\1_{y_i})\right\|_\tc<\infty.\end{equation}

The reason for which we use \eqref{E:Sum_ai} instead of the
standard condition \eqref{E:Complet} is: the conditions are
equivalent for the spaces which we consider. In fact, because all
distances in $M$ are between $1$ and $2$, we have
\[\sum_{i=s_k+1}^{s_{k+1}}
a_i\le
\left\|\sum_{i=s_k+1}^{s_{k+1}}a_i(\1_{x_i}-\1_{y_i})\right\|_\tc\le2\sum_{i=s_k+1}^{s_{k+1}}
a_i\] in the case where the transportation plan given by
$\sum_{i=s_k+1}^{s_{k+1}}a_i(\1_{x_i}-\1_{y_i})$ is optimal.

The {\it cost} of the transportation plan \eqref{E:InfTranPlan} is
defined as $\sum_{i=1}^\infty a_id(x_i,y_i)$. Observe that since
$d(x_i,y_i)\le 2$ for all $x_i$ and $y_i$, this cost is always
finite if $\sum_{i=1}^\infty a_i<\infty$. The definition of a
completion also implies that $\|f\|_\tc$ is the infimum of costs
of generalized transportation plans for $f$ for every
$f\in\tc(M)$.

It turns out that  in  this metric space $M$, for each
$f\in\tc(M)$, there exists a minimum-cost generalized
transportation plan, which can be described in the following way.
Denote by $f_i$ the value of $f\in\tc(M)$ at $i$. Let
$m\in\mathbb{N}$ be such that $\sum_{i=1}^m|f_i|\ge \|f\|_1/2$ and
$\sum_{i=1}^{m-1}|f_i|<\|f\|_1/2$. We call $m$ the {\it median} of
the support of $f$. We represent $f$ as a sum of ``beginning'' and
``end'', namely
\[b=\sum_{i=1}^{m-1}f_i\1_i+\sign f_m\left(\frac{\|f\|_1}2-\sum_{i=1}^{m-1}|f_i|\right)\1_m\]
and
\[e=\sum_{i=m+1}^{\infty}f_i\1_i+\sign f_m\left(\frac{\|f\|_1}2-\sum_{i=m+1}^{\infty}|f_i|\right)\1_m.\]

Observe that

\begin{equation}\label{E:Cond_b&e} b+e=f\quad\hbox{and}\quad
\|b\|_1=\|e\|_1=\|f\|_1/2.\end{equation}

\begin{lemma}\label{L:Opt} A generalized transportation plan
$f=\sum_{i=1}^\infty a_i(\1_{x_i}-\1_{y_i})$ with $a_i>0$ is
optimal if the following conditions are satisfied:

\begin{itemize}

\item[{\bf (1)}] All $x_i,y_i$ are in the support of $f$.

\item[{\bf (2)}] The signs of the function
$a_i(\1_{x_i}-\1_{y_i})$ at $x_i$ and $y_i$ are the same as the
signs of $f$ restricted to $x_i$ and $y_i$.

\item[{\bf (3)}] Out of each pair $x_i, y_i$, one point is in the
support of $b$ and the other  is in the support of $e$.
\end{itemize}
\end{lemma}

\begin{proof}
It can be readily seen that \eqref{E:Cond_b&e} implies the
existence of such plans. Also, it is easy to see that the cost of
each of them equals:

\begin{equation}\label{E:OptCost}\sum_{i=1}^{m-1}|f_i|h(i)+\left(\|f\|_1/2-\sum_{i=1}^{m-1}|f_i|\right)h(m).\end{equation}

It remains to show that the cost of a generalized transportation
plan cannot be less than \eqref{E:OptCost}. We prove this in three
steps labelled as {\bf (i)-(iii)} according to items {\bf (1)-(3)}
above, respectively.\medskip

{\bf (i)} We show that if a generalized transportation plan
\begin{equation}\label{E:PlanAlpha} \sum_{i=1}^\infty\alpha_i(\1_{x_i}-\1_{y_i})\end{equation} with $\alpha_i>0$ does not
satisfy {\bf (1)}, it can be modified in such a way that its cost
decreases and ultimately one obtains a plan satisfying {\bf (1)}.

In fact, if $t\notin\supp f$, then the series of coefficients of
$\1_t$ in \eqref{E:PlanAlpha} adds to $0$, so the part of
\eqref{E:PlanAlpha} containing $\1_t$ can be written as
\begin{equation}\label{E:Sum_with_t}\sum_{k=1}^\infty\alpha_{n_k}(\1_t-\1_{y_{n_k}})+\sum_{j=1}^\infty\alpha_{m_j}(\1_{x_{m_j}}-\1_t),
\end{equation}
where $\{n_k\}_{k=1}^\infty$ and $\{m_j\}_{j=1}^\infty$ are two
disjoint - possibly finite - subsets in $\mathbb{N}$ and
$\sum_{k=1}^\infty\alpha_{n_k}=\sum_{j=1}^\infty\alpha_{m_j}=\alpha$.

It can be noticed that the sum \eqref{E:Sum_with_t} admits the
representation:
\begin{equation}\label{E:2ndSum_with_t}\sum_{j,k=1}^\infty\nu_{j,k}(\1_{x_{m_j}}-\1_{y_{n_k}}),
\end{equation}
where $\nu_{j,k}\ge 0$ and $\sum_{j,k=1}^\infty\nu_{j,k}=\alpha$.
From here, one derives that
\[\sum_{k=1}^\infty\alpha_{n_k}d(t,y_{n_k})+\sum_{j=1}^\infty\alpha_{m_j}
d(x_{m_j},t)>\sum_{j,k=1}^\infty\nu_{j,k} d(x_{m_j},y_{n_k}).\]
This is because all distances are in the open interval $(1,2)$.

The procedure may be repeated for all points $t$ violating {\bf
(1)} and in the limit we get a generalized transportation plan
satisfying {\bf (1)} whose cost does not exceed the cost of the
original plan. Therefore, it may be assumed that
\eqref{E:PlanAlpha} satisfies condition {\bf (1)} of Lemma
\ref{L:Opt}.
\medskip

{\bf (ii)} Suppose that the plan \eqref{E:PlanAlpha} does not
satisfy {\bf (2)}. Let $q$ be a point at which the condition is
not satisfied, implying  that $f_q\ne 0$ and that $\1_q$ is
present in  three nonzero sums:
\begin{equation}\label{E:Sum_with_q}\sum_{k=1}^\infty\alpha_{n_k}(\1_q-\1_{y_{n_k}})+\sum_{j=1}^\infty\alpha_{m_j}(\1_{x_{m_j}}-\1_q)+
\sum_{l=1}^\infty\alpha_{s_l}\sign f_q(\1_q-\1_{z_{s_l}}),
\end{equation}
where $\{n_k\}$ and $\{m_j\}$ are disjoint,
$\sum_{k=1}^\infty\alpha_{n_k}=\sum_{j=1}^\infty \alpha_{m_j}$,
$\sum_{l=1}^\infty \alpha_{s_l}=|f_q|$, $z_{s_l}=y_{s_l}$ if
$\sign f_q=1$, while  $z_{s_l}=x_{s_l}$ if $\sign f_q=-1$.

Now we modify the sum
$\sum_{k=1}^\infty\alpha_{n_k}(\1_q-\1_{y_{n_k}})+\sum_{j=1}^\infty\alpha_{m_j}(\1_{x_{m_j}}-\1_q)$
in the transportation plan exactly in the same way as in {\bf
(i)}, and get a cheaper plan, where the condition {\bf (2)} is
satisfied for $q$, and no violators of conditions {\bf (1)} or
{\bf (2)} are added.

Repeating the procedure for all points $q$ violating {\bf (2)},
 in the limit, we reach a generalized transportation plan
satisfying {\bf (2)} and {\bf (1)}, whose cost does not exceed the
cost of the original plan. Hence, one may assume that
\eqref{E:PlanAlpha} satisfies {\bf (1)} and {\bf (2)} of Lemma
\ref{L:Opt}.
\medskip

{\bf (iii)} Assume the contrary to {\bf (3)}. Let $\sum_{i\in
A}\alpha_i(\1_{x_i}-\1_{y_i})$ be the sum of all terms of the
generalized transportation plan in which both $x_i$ and $y_i$ are
$\le m$ and $\sum_{i\in B}\alpha_i(\1_{x_i}-\1_{y_i})$ be the sum
of all terms of the generalized transportation plan in which both
$x_i$ and $y_i$ are $\ge m$.

Let us show that conditions \eqref{E:Cond_b&e} imply that
\begin{equation}\label{E:A=B}\sum_{i\in A}\alpha_i=\sum_{i\in
B}\alpha_i.\end{equation}

In fact, the sum $\sum_{i\in \mathbb{N}\backslash(A\cup
B)}\alpha_i(\1_{x_i}-\1_{y_i})$ contributes equally to the
$\ell_1$-norm of both $b$ and $e$. Therefore, by
\eqref{E:Cond_b&e}, the remaining contributions should be also
equal yielding \eqref{E:A=B} due to the fact that the
transportation plan satisfies {\bf (1)} and {\bf (2)}.

Condition \eqref{E:A=B} implies that we can redesign the
generalized transportation plan \eqref{E:PlanAlpha} in such a way
that the product is moved from $x_i$, $i\in A$, to $y_i$, $i\in
B$, and from $x_i$, $i\in B$ to $y_i$, $i\in A$. As a result we
get a cheaper generalized transportation plan satisfying
conditions {\bf (1)-(3)} of Lemma \ref{L:Opt}.
\end{proof}

Now, suppose that there is a subspace of $\tc(M)$ isometric to
$\ell_1$, and let vectors $\{x_n\}_{n=1}^\infty\subset\tc(M)$ be
isometrically equivalent to the unit vector basis of $\ell_1$.
\medskip

\begin{lemma} The vectors $\{x_n\}$ have disjoint supports.
\end{lemma}

\begin{proof} Suppose that $i$ is in the support of both $x_n$ and
$x_p$. We may assume that the signs of $x_n(i)$ and $x_p(i)$ are
different, changing $x_n$ to $-x_n$, if needed. Assume $x_n(i)<0$
and $x_p(i)>0$

It suffices to show that
$\|x_n+x_p\|_\tc<\|x_n\|_\tc+\|x_p\|_\tc$, as it leads to a
contradiction.

To achieve this, it is enough to establish that the sum of
minimum-cost generalized transportation plans for $x_n$ and $x_p$
is not a minimum-cost transportation plan for $x_n+x_p$ since it
can be improved. This can be seen as follows: since $x_n(i)<0$ and
$x_p(i)>0$, in the sum of minimum-cost generalized transportation
plans for $x_n$ and $x_p$, we deliver $-x_n(i)$ units to $i$ and
move $x_p(i)$ units from $i$. It is clear that we can move
$\min\{-x_n(i), x_p(i)\}$ units directly, and since all triangle
inequalities in $M$ are strict, this will decrease the cost of the
plan.
\end{proof}

 Let $m_1$ be the median of the support of
$x_1=\{x_{1,i}\}_{i=1}^\infty$. For $j\in\mathbb{N}$, denote by
$[1,j]$ the interval $\{1,\dots,j\}$ of integers. Observe that it
is impossible that $\supp x_1\subset [1,m_1]$. Indeed, this would
imply that $|x_{1,m_1}|>\|x_1\|_1/2$, which cannot happen for a
function with zero sum. Hence, there are elements in $\supp x_1$
which are larger than $m_1$. Let $k$ be the least such element.

Since $\{x_n\}$ are disjointly supported, there exists $x_p$ such
that all elements of the support of $x_p$ are larger than $k$. We
assert that in this case
$\|x_1+x_p\|_\tc<\|x_1\|_\tc+\|x_p\|_\tc$, getting a contradiction
with the assumption that $\{x_n\}$ is isometrically equivalent to
the unit vector basis of $\ell_1$. Denote by $m_p$ the median of
the support of $x_p$ and by $m_{+}$ the median of the support of
$x_1+x_p$.

Let us analyze the relations between the optimal transportation
plans for $x_1$, $x_p$, and $x_1+x_p$. Since $x_1$ and $x_p$ are
disjointly supported, in the optimal transportation plan for
$x_1+x_p$ we have to move $\|x_1\|_1/2+\|x_p\|_1/2$ units of
product.

Those $\|x_1\|_1/2$ units of product in $x_1$, which were
located/needed in the lower half of support of $x_1$, both in the
plan for $x_1$ and in the plan for $x_1+x_p$, will be moved
to/from locations corresponding to larger elements of
$\mathbb{N}$, more precisely, to some locations corresponding to
the upper half of support of $x_1$ and some locations
corresponding to the upper half of support of $x_1+x_p$,
respectively. Because the distance $d(i,j)$, $i\ne j$, depends
only on $\min\{i,j\}$, the cost of these relocations in both cases
will be $\|x_1\|_\tc$.
\medskip

After that, in the optimal transportation plan for $x_1+x_p$ we
need to pick the ``next'' $\|x_p\|_1/2$ units of product of
$x_1+x_p$ located between $m_1$ and $m_+$ (possibly inclusive) and
move them from/to for distances $h(i)$ corresponding to their
locations.

Observe that we do almost the same in the optimal transportation
plan for $x_p$, but there, of course, we pick only units
corresponding to $x_p$.

Since both $m_1$ and $k$ are less than any element of $\supp x_p$,
in the first case some of the locations corresponding to these
$\|x_p\|_1/2$ units for the optimal transportation plan for
$x_1+x_p$ will be strictly smaller than the locations for lower
$\|x_p\|_1/2$ units of $x_p$.

Since $h(i)$ is a strictly increasing function, this implies that
the cost of relocation of these $\|x_p\|_1/2$ units in the optimal
plan for $x_1+x_p$ is strictly smaller than $\|x_p\|_\tc$.
\end{proof}

\section{On the kernel of the seminorm
$\|\cdot\|_{\mathcal{DP}}$}\label{S:KerDP}

In \cite{MPV08} the seminorm $\|\cdot\|_{\mathcal{DP}}$ is called
the {\it double-point norm}. However, it appears
 that, for some metric spaces $M$, the
seminorm $\|\cdot\|_{\mathcal{DP}}$ is not a norm.

\begin{observation}\label{O:DPNotNorm} The seminorm $\|\cdot\|_{\mathcal{DP}}$ is not a
norm if and only if there exists a nonzero function $f\in \tp(M)$
such that
\begin{equation}\label{E:DPnotnorm}
\sum_{x\in M}f(x) d(v,x) \;\; does\;\; not\;\; depend\;\; on\;\;
v.
\end{equation}
\end{observation}

\begin{proof} In fact, $\|f\|_{\mathcal{DP}}$ is the supremum over $u$
and $v$ of
\[\left|\sum_{x\in M}\frac{d(v,x)-d(u,x)}2\,f(x)\right|=\frac12\left|
\sum_{x\in M}f(x) d(v,x)-\sum_{x\in M}f(x)
d(u,x)\right|.\qedhere\]
\end{proof}

Examples of such metric spaces $M$ are provided below.

\begin{example}\label{Ex:KerDP}
Let $M$ be a $4$-cycle and
\begin{equation}\label{E:FunInKer} f=\1_{x_1}-\1_{x_2}+\1_{x_3}-\1_{x_4},\end{equation} where
$x_1,x_2,x_3,x_4$ are vertices of the cycle in the cyclic order.
Then, condition \eqref{E:DPnotnorm} is satisfied and hence
$\|f\|_{\mathcal{DP}}=0$ and $\|\cdot\|_{\mathcal{DP}}$ is not a
norm.
\end{example}

 \begin{example} The preceding example can be generalized to a sufficient condition for existence of $f\ne 0$ with
$\|f\|_{\mathcal{DP}}=0$.  The condition is the existence in $M$
of a $4$-tuple $x_1,x_2,x_3,x_4$, such that

\[d(x_1,x_2)=d(x_2,x_3)= d(x_3,x_4)=d(x_4,x_1),\]

\[d(x_1,x_3)=d(x_2,x_4)=2d(x_1,x_2),\]
and, in addition, for each $x\in M${,} we have:
\[d(x,x_1)+d(x,x_3)=d(x,x_2)+d(x,x_4).\]

In this case  function \eqref{E:FunInKer} also satisfies
$\|f\|_{\mathcal{DP}}=0$.
\end{example}

\medskip

\begin{example}\label{Ex:DPnonMetr} Another class of metric spaces for which $\|\cdot\|_\mdp$
is not norm can be constructed in the following way. \end{example}

Given $m\in\mathbb{N}$, we construct a metric space $M$ of
cardinality $2m$ as a union of two disjoint sets, $A$ and $B$
satisfying $|A|=|B|=m$. The metric on $M$ is defined as

\[d(x,y)=\begin{cases} 0\quad &\hbox{ if } x=y\\
a\quad &\hbox{ if } x,y\in A,~x\ne y\\
a\quad &\hbox{ if } x,y\in B,~x\ne y\\
c\quad &\hbox{ if } x\in A,~y\in B.\end{cases}\]

For $d(x,y)$ to be a metric it is necessary and sufficient that
\begin{equation}\label{E:Tr} 0<a\le 2c\end{equation}

Let $f=\1_A-\1_B\in\tp(M)$, where for  a subset $U\subseteq M$,
its indicator  is:

\[\1_U(x)=\begin{cases} 1 &\hbox{ if }x\in U\\
0 &\hbox{ if }x\notin U.\end{cases}\]

Then,  with a suitable choice $m,a$, and $c$, for every $v\in M$,
\begin{equation}\label{E:Desired}
\sum_{x\in M} f(x)d(v,x)=0,
\end{equation}
and the conclusion follows from Observation \ref{O:DPNotNorm}. For
$v\in A$ or $v\in B$, equation \eqref{E:Desired} becomes
\begin{equation}\label{E:InA}
(m-1)a-mc=0.
\end{equation}

Meanwhile, equation \eqref{E:InA} can be written as
\begin{equation}\label{E:a}
a=\left(\frac{m}{m-1}\right)c.
\end{equation}
Consequently, if $c>0$ is arbitrary, $m$ is any integer satisfying
$m\ge 2$,  and $a$ is given by \eqref{E:a}, condition \eqref{E:Tr}
is satisfied. Therefore, with such selection of the parameters,
\eqref{E:Desired} holds.

\section{The spaces $\tp_\mdp(T)$ when $T$ is a finite
tree}\label{S:DPforTree}

It is well known that the space $\tp_\ml(T)$ for a finite tree $T$
is isometric to $\ell_1^d$ of the corresponding dimension, see
\cite{God10} and \cite[Proposition 2.1]{DKO18+}. In this section
we show that $\tp_\mdp(T)$ is quite different, and that it changes
in the ``undesirable'' direction, as we are mostly interested, see
Section \ref{S:Directions}, in moving ``away from
$\ell_\infty^n$''.

\begin{proposition}
\label{P:TreeDP} If $T$ is a finite tree, possibly weighted, then
\begin{equation}\label{E:DPonTrees}d_{BM}(\tp_{\mdp}(T),\ell_\infty^{|E(T)|})\le
4.\end{equation}
\end{proposition}

\begin{proof} Assume that $T$ is a rooted tree and, for each its edge $e$,
consider the function $f_e:=\1_w-\1_z\in\tp(T)$, where $w$ and $z$
are the ends of $e$, and $w$ is the one closer to the root. Then,
every $f\in\tp(M)$ can be written in the form:
\begin{equation}\label{E:Repf} f=\sum_{e\in E(T)} a_ef_e.\end{equation} The norm on $\tp_\mdp(T)$ can be
calculated in the following way: Consider a path $P$ in $T$ with
ends $u$ and $v$. Suppose $P$ is directed in such a way that first
$P$ goes towards the root, denote this first part by $P_1$, and
then goes away from the root, denote the second part by $P_2$.
Set:
\[P(f)=\sum_{e\in P_1}d_ea_e-\sum_{e\in P_2}d_ea_e,\]
where $d_e$ is the weight (length) of the edge $e$. Then
\begin{equation}\label{E:DPinPaths}\|f\|_\mdp=\max_{P}|P(f)|.\end{equation}
To justify \eqref{E:DPinPaths}, observe that

\[\|f\|_\mdp=\max_{u,v\in M}
\left|\sum_{x\in M}\frac{d(v,x)-d(u,x)}2\,f(x)\right|.\] Let
\[g(x)=\frac{d(v,x)-d(u,x)}2,\]
and let $w$ and $z$ be the ends of an edge $e$. Then,
\[g(w)-g(z)=\begin{cases} 0 &\hbox{ if } e \hbox{ is not in }P,\\
d_e &\hbox{ if } e \hbox{ is in }P \hbox{ and } w \hbox{ is closer
to }u.\end{cases}\] Combining this formula with \eqref{E:Repf},
one derives \eqref{E:DPinPaths}.
\medskip

By \eqref{E:DPinPaths}, the norm $\|f\|_\mdp$, up to a factor of
$2$, is equivalent to
\[\max\{|P(f)|: P\hbox{ is a descending path in }T\},\]
and up to a factor of $4$, the norm  $\|f\|_\mdp$ is equivalent to
the next one:
\[\max\{|P(f)|: P\hbox{ is a path in }T\hbox{ with root being one of its ends}\}.\]

Let us assign to each edge $g$ the real number $s_g$ defined as
the sum of numbers $d_ea_e$ over all edges connecting $g$ with the
root, including $g$. Clearly, this defines a bijective linear map
$D$ from $\tp(T)$ to the space of real-valued functions on the
edge set $E(T)$. The discussion above implies that
\[\frac14\|f\|_\mdp\le \|Df\|_{\ell_\infty^{|E(T)|}}\le
\|f\|_\mdp,\] and in this way proves \eqref{E:DPonTrees}.
\end{proof}

\section{Locally finite representing subsets in Banach spaces}\label{S:Repr}

To begin with, let us recollect the following definition given in
\cite{DL08}.

\begin{definition}\label{D:LocFinRep} A subset $K$ of a
separable Banach space $X$ is said to be a {\it representing
subset} of $X$ if every Banach space containing an isometric copy
of $K$ contains an isometric copy of $X$.
\end{definition}

In this connection, the following problem arises.

\begin{problem}\label{P:LocFinRepr} Characterize Banach spaces for which there exist locally finite representing
sets.
\end{problem}

This problem is motivated by the applications which we mention at
the end of this section. Below, we solve this problem for $\ell_1$
by proving the following analogue of \cite[Proposition 4.3]{DL08}.

\begin{proposition}\label{P:LocFinRep} There exist locally finite metric spaces
representing $\ell_1$.
\end{proposition}

\begin{proof} Let $\{e_i\}_{i=1}^\infty$ be the unit vector basis
of $\ell_1$. Let $M$ be the subset of $\ell_1$ consisting of all
vectors of the form $\sum_{i\in A}2^ie_i$, where $A$ is a finite
subset of $\mathbb{N}$; we assume that $\sum_{i\in
\emptyset}2^ie_i=0$. We endow this subset with the $\ell_1$-metric
and consider it as a metric space. Obviously, this is a locally
finite metric space.

Suppose that $T$ is an isometric embedding of the metric space $M$
into a Banach space $X$. Without loss of generality assume that
$T(0)=0$. Let $f_i=2^{-i}T(2^ie_i)$. It is easy to see that these
vectors should have norm $1$. Our goal is to show that they are
isometrically equivalent to the unit vector basis of $\ell_1$. To
achieve this goal it suffices to prove that, for each finite
collection $\Theta=\{\theta_i\}_{i=1}^n$ with $\theta_i=\pm 1$,
there exists a normalized linear functional $F_\Theta\in X^*$ such
that $F_\Theta(f_i)=\theta_i$ for $i=1,\dots,n$. Let
$A_+=\{i\in\{1,\dots, n\}: \theta_i=1\}$ and $A_-=\{i\in\{1,\dots,
n\}: \theta_i=-1\}$. Then $x_+=\sum_{i\in A_+}2^ie_i$ and
$x_-=\sum_{i\in A_-}2^ie_i$ are in $M$, whence we have:
\begin{equation}\label{E:NormVec}\|x_-\|_1=\|T(x_-)\|_X=\sum_{i\in A_-}2^i,\quad
\|x_+\|_1=\|T(x_+)\|=\sum_{i\in A_+}2^i,\end{equation} and
\[\|x_+-x_-\|_1=\|T(x_+)-T(x_-)\|_X=\sum_{i=1}^n2^i.\] Thus, there exists $F\in X^*$,
$\|F\|=1$, such that $F(T(x_+))-F(T(x_-))=\sum_{i=1}^n2^i$. By
\eqref{E:NormVec}, this implies $F(T(x_+))=\sum_{i\in A_+}2^i$ and
$F(T(x_-))=-\sum_{i\in A_-}2^i$.
\medskip

Next, let us verify that

\[F(f_j)=\begin{cases} 1 & \hbox{~ if~} j\in A_+\\
-1 & \hbox{~ if~} j\in A_-,\end{cases}\] and, therefore, $F$ is
the desired functional $F_\Theta$. \medskip

Consider $j\in A_+$ (the case where $j\in A_-$ is similar).
Observe that $2^je_j$ is on a geodesic joining $0$ and $x_+$ in
the sense that \[\|2^je_j\|_1+\|x_+-2^je_j\|_1=\|x_+\|_1.\] Since
$T$ is an isometry and $T(0)=0$, we get that
\[\|T(2^je_j)\|_X+\|T(x_+)-T(2^je_j)\|_X=\|T(x_+)\|_X.\]
Thus,  $F(T(2^je_j))=2^j$ and $F(f_j)=1$.
\end{proof}

\begin{corollary}\label{C:L1Rep} If a Banach space $X$ contains
$\ell_1^n$ isometrically for each $n\in\mathbb{N}$, but does not
contain $\ell_1$ isometrically, then there exists a locally finite
metric space $M$ such that $X$ contains isometrically each finite
subset of $M$, but does not contain an isometric copy of $M$.
\end{corollary}

\begin{proof} Let $M$ be the locally finite subset of $\ell_1$ constructed in the proof of Proposition
\ref{P:LocFinRep}. It is clear that each finite subset of $M$ is
isometric to a subset of $\ell_1^n$ for sufficiently large $n$.
Thus, the Banach space $X$ contains isometrically every finite
subset of $M$. On the other hand, by Proposition
\ref{P:LocFinRep}, $X$ does not contain an isometric copy of $M$.
\end{proof}

{\bf Examples} of spaces satisfying the conditions of Corollary
\ref{C:L1Rep}: $c_0$, $c(\alpha)$, where $\alpha$ is a countable
ordinal, direct sums $(\oplus_{n=1}^\infty\ell_1^n)_p,
(\oplus_{n=1}^\infty\ell_\infty^n)_p$ for $1<p<\infty$. The
necessary definitions can be found in \cite{LT73}.

The spaces $c(\alpha)$, where $\alpha$ is a countable ordinal, and
$(\oplus_{n=1}^\infty\ell_1^n)_p$ for $1<p<\infty$, are new
examples of Banach spaces, for which there exists a locally finite
metric space $M$ such that $X$ contains isometrically each finite
subset of $M$, but does not contain $M$ isometrically. Previously
known examples are available in \cite[Theorem 2.9]{KL08},
\cite{OO19}, \cite{OO19a}.

\section*{Acknowledgement}

The second-named author gratefully acknowledges the support by
National Science Foundation grant NSF DMS-1700176. We would like
to thank the referee for the valuable suggestions and corrections.


\begin{small}

\renewcommand{\refname}{
\end{small}

\textsc{Department of Mathematics, Atilim University, 06830
Incek,\\ Ankara, TURKEY} \par \textit{E-mail address}:
\texttt{sofia.ostrovska@atilim.edu.tr}\par\medskip

\textsc{Department of Mathematics and Computer Science, St. John's
University, 8000 Utopia Parkway, Queens, NY 11439, USA} \par
  \textit{E-mail address}: \texttt{ostrovsm@stjohns.edu} \par


\section*{References}}

\begin{thebibliography}{99}

\bibitem[AE56]{AE56} R.\,F.~Arens, J.~Eells, Jr., On embedding uniform and topological spaces, {\it Pacific J. Math.},
{\bf  6}  (1956), 397--403.

\bibitem[BL08]{BL08} F.~Baudier, G.~Lancien, Embeddings of
locally finite metric spaces into Banach spaces, {\it Proc. Amer.
Math. Soc.}, {\bf 136} (2008), 1029--1033.

\bibitem[Bil68]{Bil68} P.~Billingsley, {\it Convergence of probability measures}. John Wiley \&\ Sons, Inc., New York-London-Sydney,
1968.

\bibitem[Blu53]{Blu53} L.\,M.~Blumenthal, {\it Theory and applications of distance
geometry}. Oxford, Clarendon Press, 1953.


\bibitem[CD16]{CD16} M.~C\'uth, M.~Doucha, Lipschitz-free spaces over ultrametric spaces. {\it Me\-di\-terr. J. Math.} {\bf 13} (2016), no. 4,
1893--1906.

\bibitem[CDW16]{CDW16} M.~C\'uth, M.~Doucha, P.~Wojtaszczyk, On the
structure of Lipschitz-free spaces. {\it Proc. Amer. Math. Soc.}
{\bf 144} (2016), no. 9, 3833--3846.

\bibitem[CJ17]{CJ17} M.~C\'uth, M.~Johanis,
Isometric embedding of $\ell_1$ into Lipschitz-free spaces and
$\ell_\infty$ into their duals. {\it Proc. Amer. Math. Soc.} {\bf
145} (2017), no. 8, 3409--3421.

\bibitem[Dal15]{Dal15} A.~Dalet, Free spaces over some proper metric spaces.
{\it Mediterr. J. Math.} {\bf 12} (2015), no. 3, 973--986.

\bibitem[DKO18+]{DKO18+} S.\,J. Dilworth, D.~Kutzarova,
M.\,I.~Ostrovskii, Lipschitz-free spaces on finite metric spaces,
{\it Canad. J. Math.}, to appear, {\tt arXiv:1807.03814}.

\bibitem[DL08]{DL08} Y.~Dutrieux, G.~Lancien,  Isometric embeddings of compact
spaces into Banach spaces. {\it J. Funct. Anal.} {\bf 255} (2008),
no. 2, 494--501.

\bibitem[Fre10]{Fre10} M.~Fr\'echet, Les dimensions d'un ensemble abstrait, {\it Math. Ann.},
{\bf  68}  (1910),  no. 2, 145--168.

\bibitem[Gar18]{Gar18} D.\,J.\,H.~Garling, {\it Analysis on Polish spaces and an introduction
to optimal transportation}. London Mathematical Society Student
Texts, {\bf 89}. Cambridge University Press, Cambridge, 2018.

\bibitem[God10]{God10}
A.~Godard,  Tree metrics and their Lipschitz-free spaces, {\it
Proc. Amer. Math. Soc.}, {\bf 138} (2010), no. 12, 4311--4320.

\bibitem[GK03]{GK03} G.~Godefroy, N.\,J.~Kalton, Lipschitz-free Banach spaces, {\it Studia Math.},
{\bf 159}  (2003),  no. 1, 121--141.

\bibitem[Kad85]{Kad85} V.\,M.~Kadets, Lipschitz mappings of metric spaces
(Russian), {\it Izv. Vyssh. Uchebn. Zaved. Mat.} {\bf 1985}, no.
1, 30--34; English transl.: {\it Soviet Math.} (Iz. VUZ), {\bf 29}
(1985), no. 1, 36--41.

\bibitem[KL08]{KL08} N.\,J.~Kalton, G.~Lancien, Best constants for Lipschitz embeddings of
metric spaces into $c_0$, {\it Fund. Math.}, {\bf 199} (2008),
249--272.

\bibitem[Kan42]{Kan42} L.\,V.~Kantorovich, On mass transportation
(Russian), {\it Doklady Akad. Nauk SSSR}, (N.S.) {\bf 37}, (1942),
199--201; English transl.: {\it J. Math. Sci.} (N. Y.), {\bf  133}
(2006), no. 4, 1381--1382.

\bibitem[Kan11]{Kan11} L.\,V.~Kantorovich, {\it Mathematical-economic articles. Selected works} (Russian),
 Nauka, Novosibirsk, 2011.


\bibitem[KA82]{KA82} L.\,V.~Kantorovich, G.\,P.~Akilov,  {\it Functional analysis}.
Translated from the Russian by Howard L. Silcock. Second edition.
Pergamon Press, Oxford-Elmsford, N.Y., 1982.

\bibitem[KA84]{KA84} L.\,V.~Kantorovich, G.\,P.~Akilov, {\it Functional
analysis}
(Russian), Third edition, Nauka, Moscow, 1984.

\bibitem[KG49]{KG49} L.\,V.~Kantorovich, M.\,K.~Gavurin, Application of mathematical methods in the analysis of cargo flows (Russian),
in: {\it Problems of improving of transport efficiency}, USSR
Academy of Sciences Publishers, Moscow, 1949, pp. 110--138.

\bibitem[KR57]{KR57} L.\,V.~Kantorovich, G.\,S.~Rubinstein,
On a functional space and certain extremum problems (Russian),
{\it Dokl. Akad. Nauk SSSR} (N.S.), {\bf 115} (1957), 1058--1061.

\bibitem[KR58]{KR58} L.\,V.~Kantorovich, G.\,S.~Rubinstein,
On a space of completely additive functions (Russian), {\it
Vestnik Leningrad. Univ.}, {\bf  13}  (1958), no. 7, 52--59.

\bibitem[KK04]{KK04} A.\,G.~Kusraev, S.\,S.~Kutateladze, Kantorovich spaces and
optimization. {\it Zap. Nauchn. Sem. S.-Peterburg. Otdel. Mat.
Inst. Steklov.} (POMI) {\bf 312} (2004), {\it Teor. Predst. Din.
Sist. Komb. i Algoritm. Metody}. {\bf 11}, 138--149, 313--314;
English translation in {\it J. Math. Sci.} (N.Y.) {\bf 133}
(2006), no. 4, 1449--1455.

\bibitem[LT73]{LT73} J.~Lindenstrauss, L.~Tzafriri, {\it Classical Banach
spaces.} Lecture Notes in Mathematics, Vol. {\bf 338}.
Springer-Verlag, Berlin-New York, 1973.

\bibitem[Mar41]{Mar41} A.~Markov, On free topological groups. {\it Doklady Akad.
Nauk SSSR}, {\bf 31} (1941), 299--301.

\bibitem[Mar45]{Mar45} A.~Markov, On free topological groups (Russian), {\it Izvestiya Akad. Nauk SSSR},
{\bf 9} (1945), 3--64; English translation in: {\it Translations},
Ser. 1, Vol. 8: {\it Topology and topological algebra}. American
Mathematical Society, Providence, R.I. 1962, pp. 195--272.

\bibitem[Mau03]{Mau03} B.~Maurey, Type, cotype and $K$-convexity. {\it Handbook of the
geometry of Banach spaces}, Vol. 2, 1299--1332, North-Holland,
Amsterdam, 2003.

\bibitem[MPV08]{MPV08} J.~Melleray,  F.\,V.~Petrov, A.\,M.~Vershik, Linearly rigid
metric spaces and the embedding problem, {\it Fund. Math.}, {\bf
199} (2008), no. 2, 177--194.

\bibitem[Mic64]{Mic64} E.~Michael, A short proof of the Arens-Eells embedding theorem.
{\it Proc. Amer. Math. Soc.}, {\bf 15} (1964), 415--416.

\bibitem[Nao18]{Nao18} A.~Naor, Metric dimension reduction: a snapshot of the Ribe program, {\it Proc. Int. Cong. of
Math. - 2018}, Rio de Janeiro, Vol. {\bf 1}, 759--838.

\bibitem[OO19]{OO19} S.~Ostrovska, M.\,I.~Ostrovskii,
Distortion in the finite determination result for embeddings of
locally finite metric spaces into Banach spaces, {\it Glasg. Math.
J.}, {\bf 61} (2019), no. 1, 33--47.

\bibitem[OO19a]{OO19a} S.~Ostrovska, M.\,I.~Ostrovskii, On embeddings of locally finite metric spaces into
$\ell_p$, {\it  J. Math. Anal. Appl.} {\bf 474} (2019), 666--673.

\bibitem[Ost12]{Ost12} M.\,I.~Ostrovskii,
Embeddability of locally finite metric spaces into Banach spaces
is finitely determined, {\it Proc. Amer. Math. Soc.}, {\bf 140}
(2012), 2721--2730.

\bibitem[Ost13]{Ost13} M.\,I.~Ostrovskii, {\it Metric Embeddings: Bilipschitz and Coarse Embeddings into Banach Spaces},
de Gruyter Studies in Mathematics, {\bf 49}. Walter de Gruyter \&\
Co., Berlin, 2013.

\bibitem[Ost13a]{Ost13b} M.\,I.~Ostrovskii, Different forms of metric
characterizations of classes of Banach spaces, {\it Houston. J.
Math.}, {\bf 39} (2013), no. 3, 889--906.

\bibitem[Pes86]{Pes86} V.\,G.~Pestov, Free Banach spaces and representations of
topological groups (Russian), {\it Funktsional. Anal. i
Prilozhen.}, {\bf 20} (1986), no. 1, 81--82; English transl.: {\it
Funct. Anal. Appl.} {\bf 20} (1986), 70--72.

\bibitem[RR98]{RR98} S.\,T.~Rachev, L.~R\"uschendorf,
{\it Mass transportation problems}. Vol. I. Theory. Probability
and its Applications. Springer-Verlag, New York, 1998.

\bibitem[RTG98]{RTG98} Y.~Rubner, C.~Tomasi, L.\,J.~Guibas, A
metric for distributions with applications to image databases,
{\it Proceedings ICCV 1998}, pp.~59--66;\\
doi:10.1109/ICCV.1998.710701.

\bibitem[Shi54]{Shi54} M.~Shimrat, Embedding in homogeneous spaces. {\it Quart. J. Math.},
Oxford Ser. (2) {\bf 5} (1954), 304--311.

\bibitem[Ver04]{Ver04} A.\,M.~Vershik, The Kantorovich metric: the initial history and
little-known applications. (Russian) {\it Zap. Nauchn. Sem.
S.-Peterburg. Otdel. Mat. Inst. Steklov.} (POMI) {\bf 312} (2004),
Teor. Predst. Din. Sist. Komb. i Algoritm. Metody. {\bf 11},
69--85, 311; translation in J. Math. Sci. (N.Y.) {\bf 133} (2006),
no. 4, 1410--1417.

\bibitem[Ver13]{Ver13} A.\,M.~Vershik, Long history of the Monge-Kantorovich transportation problem.
{\it Math. Intelligencer} {\bf 35} (2013), no. 4, 1--9.

\bibitem[Vas69]{Vas69} L.\,N.~Vasershtein, Markov processes over denumerable products
of spaces describing large system of automata. {\it Problems of
Information Transmission} {\bf 5} (1969), no. 3, 47--52;
translated from: {\it Problemy Peredachi Informatsii} {\bf 5}
(1969), no. 3, 64--72.

\bibitem[Vil03]{Vil03} C.~Villani, {\it Topics in optimal transportation.} Graduate Studies in Mathematics,
{\bf 58}. American Mathematical Society, Providence, RI, 2003.

\bibitem[Vil09]{Vil09} C.~Villani,
{\it Optimal transport.} Old and new. Grundlehren der
Mathematischen Wissenschaften, {\bf 338}. Springer-Verlag, Berlin,
2009.

\bibitem[Wea99]{Wea99} N.~Weaver, {\it Lipschitz algebras}, World Scientific Publishing
Co., Inc., River Edge, NJ, 1999.

\bibitem[Wea18]{Wea18} N.~Weaver, {\it Lipschitz algebras}, Second edition, World Scientific
Publishing Co. Pte. Ltd., Hackensack, NJ, 2018.

\bibitem[Zat08]{Zat08} P.\,B.~Zatitskii,
On the coincidence of the canonical embeddings of a metric space
into a Banach space, {\it J. Math. Sci.}, New York, {\bf 158}
(2009), No. 6, 853--857; translation from {\it Zap. Nauchn. Semin.
POMI}, {\bf 360} (2008), 153--161.

\bibitem[Zat10]{Zat10} P.\,B.~Zatitskii, Canonical embeddings of compact metric
spaces. {\it Zap. Nauchn. Sem. S.-Peterburg. Otdel. Mat. Inst.
Steklov.} (POMI) {\bf 378} (2010), Teoriya Predstavlenii,
Dinamicheskie Sistemy, Kombinatornye Metody. XVIII, 40--46, 229;
translation in {\it J. Math. Sci.} (N.Y.) {\bf 174} (2011), no. 1,
19--22.


\end{thebibliography}
\end{document}